\documentclass[11pt]{amsart}
\usepackage{multirow,lineno,hyperref,amsmath,amsthm,bm,graphicx}
\usepackage{epstopdf,color}
\usepackage[top=2.54cm, bottom=2.54cm, left=3cm,right=3cm]{geometry}
\graphicspath{{pics/}}

\newtheorem{theorem}{Theorem}

\begin{document}

\title[Savage-Hutter Equations]{A Finite Volume Scheme for
  Savage-Hutter Equations on Unstructured Grids}

\author[R. Li]{Ruo Li} \address{CAPT, LMAM and School of Mathematical
  Sciences, Peking University, Beijing 100871, P.R. China}
\email{rli@math.pku.edu.cn}

\author[X.-H. Zhang]{Xiaohua Zhang} \address{College of Science and
  Three Gorges Mathematical Research Center, China Three Gorges
  University, Yichang, 443002, China}
\email{zhangxiaohua07@163.com}

\maketitle

\begin{abstract}
  A Godunov-type finite volume scheme on unstructured triangular grids
  is proposed to numerically solve the Savage-Hutter equations in
  curvilinear coordinate. We show the direct observation that the
  model is a not Galilean invariant system. At the cell boundary, the
  modified Harten-Lax-van Leer (HLL) approximate Riemann solver is
  adopted to calculate the numerical flux. The modified HLL flux is
  not troubled by the lack of Galilean invariance of the model and it
  is helpful to handle discontinuities at free interface. Rigidly the
  system is not always a hyperbolic system due to the dependence of
  flux on the velocity gradient. Even though, our numerical results
  still show quite good agreements to reference solutions. The
  simulations for granular avalanche flows with shock waves indicate
  that the scheme is applicable.

  \textbf{keywords}: granular avalanche flow, Savage-Hutter equations,
  finite volume method, Galilean invariant


\end{abstract}


\section{Introduction}

The Savage-Hutter model was first proposed to describe the motion of a
finite mass of granular material flowing down a rough incline, in
which the granular material was treated as an incompressible continuum
\cite{Savage1989, Gray2003, Hutter2005}. It was then extended to 2D
and with curvilinear coordinate \cite{Hutter1993, Greve1994,
  Universit1999} which may be applied to study natural disasters as
landslides and debris flows \cite{Paik2015}. For more details on
Savage-Hutter model, we refer to \cite{Hutter2005, Pudasaini2007,
  Gray2001, Iverson1997, Iverson2001a, Pitman2005}.

Numerical method for the Savage-Hutter model may be dated back to work
in \cite{Savage1989, Koch1994} using finite difference methods which
are not able to capture shock waves. Later on, regardless of the
Savage-Hutter model or other avalanches of granular flow models, the
Godunov-type schemes are commonly used in literature
\cite{Denlinger2001a, Wang2004, Chiou2005, Cui2014, Zhai2015,
  Vollmoller2004, Rosatti2013} considering that most of them have
hyperbolic nature of the depth averaged equations. We note that
rigidly, the Savage-Hutter model is not hyperbolic since its flux is
depended on the gradient of velocity. Precisely, the dependence is on
the sign of the gradient of velocity components, that the model does
be hyperbolic in the region with monotonic velocity components. Some
related interesting development on numerical methods can be found in
such as \cite{Tai2002, Nichita2004, Chiou2005, Pelanti2011,
  Ouyang2013}.

In practical scenario, landslides, rock avalanches and debris flows
usually occur in mountainous areas with complex topography, that we
have to use unstructured grids for numerical simulation. However,
there is few work of numerical methods for Savage-Hutter equations on
unstructured grids. This may due primarily to the lack of Galilean
invariance of the model, which is a direct observation to the model as
we point out in section 2. In this paper, we are motivated to develop
a high resolution and robust numerical model for Savage-Hutter
equations under the unstructured grids. Whether the numerical method
may provide a correct solution on the unstructured grids for a problem
without Galilean invariance is a major question here. Meanwhile, the 
loss of hyperbolicity at the zeros of velocity gradients may some
unpredictable behavior in numerical solutions, too.

Additionally, there are still some challenges in numerical solving
strongly convective Savage-Hutter equations on unstructured grids. For
example, shock formation is an essential mechanism in granular flows
on an inclined surface merging into a horizontal run-out zone or
encountering an obstacle when the velocity becomes subcritical from
its supercritical state \cite{Wang2004}. Therefore, numerical
efficiency is an important issue to help resolve the steep gradients
and moving fronts. Noticing that the Savage-Hutter equations and
Shallow water equations have some similarities, we basically follow
the method in \cite{Deng2013} for shallow water equations. The
modified HLL flux is adopted to calculate the numerical flux on the
cell boundary. The formation of the modified HLL flux does not require
the flux to be Galilean invariant, thus formally we have no
difficulties in calculation of numerical flux. The techniques involved
to improve efficiency include the MUSCL-Hancock scheme in time
discretization, the ENO-type reconstruction and the $h$-adaptive
method in spatial discretization.

We apply the numerical scheme to three examples. The first one is a
granular dam break problem, and the second one is a granular avalanche
flows down an incline plane and merges continuously into a horizontal
plane. In the third example, a granular slides down an inclined plane
merging into a horizontal run-out zone is simulated, whereby the flow
is diverted by an obstacle which is located on the inclined plane. It
is interesting that in these examples, the numerical solutions are
agreed with the references solutions quite well. We are indicated that
the scheme is applicable to the model, though the model is lack of the
Galilean invariance and is not always hyperbolic. The reason why the
loss of the hyperbolicity of the model are not destructive to the
numerical methods is still under investigation.

The rest part of the paper is organized as follows. In section 2, the
Savage-Hutter equations on the curivilinear coordinate is briefly
introduced. We directly show that the system is not Galilean
invariant. In section 3, we give the details of our numerical
scheme. In section 4, we present the numerical results and a short
conclusion in section 5 close the paper.


\section{Governing Equations}

There are various forms of Savage-Hutter equations according to
different coordinate system and a detailed derivation of these
different versions of Savage-Hutter equations has been given in
\cite{Pudasaini2007}. Here, we confine ourselves to one of
them. Precisely, the governing equations are built on an orthogonal
curvilinear coordinate system (see Fig. \ref{fig:sketch}) given in
\cite{Chiou2005}. The curvilinear coordinate $Oxyz$ is defined on the
reference surface, where the $x$-axis is oriented in the downslope,
the $y$-axis lies in the cross-slope direction of the reference
surface and the $z$-axis is normal to them. The downslope inclination
angle of the reference surface $\zeta$ only depends on the downslope
coordinates $x$, that is, there is no lateral variation in the
$y$-direction, thus it is strictly horizontal. The shallow basal
topography is defined by its elevation $z = z^{b}(x,y)$ above the
curvilinear reference surface.
\begin{figure}[htbp] 
  \centering 
  \includegraphics[width=0.4\textwidth]{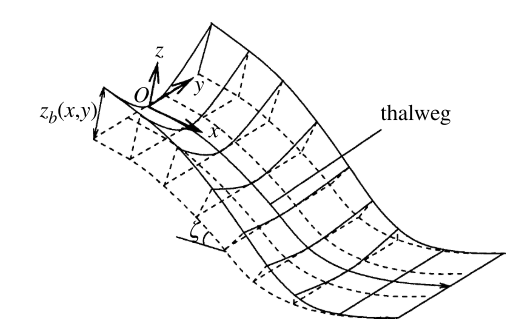} 
  \caption{The sketch of curvilinear coordinate for Savage-Hutter
    model(reproduced from \cite{Chiou2005}).}
  \label{fig:sketch} 	
\end{figure}

The corresponding non-dimensionless Savage-Hutter equations are
\begin{equation}
  \frac{\partial h}{\partial t} + \frac{\partial}{\partial x}(hu)
  + \frac{\partial}{\partial y}(hv) = 0, \label{mass}
\end{equation}
\begin{equation}
  \frac{\partial(hu)}{\partial t} + \frac{\partial}{\partial x}(hu^2
  + \frac{1}{2}\beta_x h^2) + \frac{\partial}{\partial y}(huv) = hs_x,
  \label{momentumx}
\end{equation}
\begin{equation}
  \frac{\partial (hv)}{\partial t} + \frac{\partial}{\partial x}(huv)
  + \frac{\partial}{\partial y}(hv^2 + \frac{1}{2} \beta_y h^2) =
  hs_y,
  \label{momentumy}
\end{equation}
where $h$ is the avalanche depth in the $z$ direction, and
$\bm{u} = (u, v)$ are the depth-averaged velocity components in the
downslope($x$) and cross-slope($y$) directions, respectively. The
factors $\beta_x$, $\beta_y$ are defined as
\[
  \beta_x = \epsilon \cos\zeta K_x, \quad \beta_y = \epsilon \cos\zeta
  K_y,
\]
respectively, where $\epsilon$ is aspect ratio of the characteristic
thickness to the characteristic downslope extent. Here $K_x$, $K_y$
are the $x$, $y$ directions earth pressure coefficients defined by the
Mohr-Coulomb yield criterion. Generally, the earth pressure
coefficients are considered in the active or passive state, which
depends on whether the downslope and cross-slope flows are expanding
or contracting \cite{Pirulli2007}. Hutter et al. assumed that $K_x$,
$K_y$ link the normal pressures in the $x$, $y$ directions with the
overburden pressure and suggested that \cite{Savage1991, Hutter1993}
\[
  K_{x_{act/pass}}=2\left(1\mp
    \sqrt{1-\frac{\cos^2\phi}{\cos^2\delta}}\right)\sec^2\phi - 1,
\]
\[
  K_{y_{act/pass}}=\frac{1}{2}\left(K_x + 1\mp \sqrt{(K_x -1)^2 + 4
      \tan^2 \delta}\right).
\]
where $\phi$ and $\delta$ are the internal and basal Coulomb friction 
angles, respectively. The subscripts ``act'' and ``pass'' denote 
active(``$-$'') and passive (``$+$'') stress by
\[
K_x = \left \{\begin{array}{lr}
K_{x_{act}}, & \partial u / \partial x \ge 0, \\
K_{x_{pass}}, & \partial u / \partial x < 0,
\end{array} \right.
\]
\[
  K_y = \left \{
    \begin{array}{llr}
      K_{y_{act}}^{x_{act}}, & \partial u/ \partial x \ge 0,
      & \partial v/ \partial y \ge 0,  \\ [2mm]
      K_{y_{pass}}^{x_{act}}, & \partial u/ \partial x \ge 0,
      & \partial v / \partial y < 0, \\ [2mm]
      K_{y_{act}}^{x_{pass}}, & \partial u/ \partial x < 0, & \partial
                                                              v/ \partial
                                                              y \ge 0,
      \\ [2mm]
      K_{y_{pass}}^{x_{pass}}, & \partial u/ \partial x < 0, & \partial v / \partial y < 0.
    \end{array}\right.
\]
In this model the earth pressure coefficients $K_x$, $K_y$ are assumed
to be functions of the velocity gradient.

The terms $s_x$, $s_y$ are the net driving accelerations in the $x$,
$y$ directions, respectively,
\[
  s_x = \sin\zeta - \frac{u}{|\bm{u}|} \tan\delta (\cos\zeta +\lambda
  \kappa u^2) - \epsilon \cos\zeta \frac{\partial z_b}{\partial x},
\]
\[
  s_y = -\frac{v}{|\bm{u}|}\tan \delta (\cos\zeta + \lambda \kappa
  u^2) - \epsilon \cos\zeta \frac{\partial z_b}{\partial y},
\]
where $|\bm{u}| = \sqrt{u^2 + v^2}$,
$\kappa = -\frac{\partial \zeta}{\partial x}$ is the local curvature
of the reference surface, and $\lambda \kappa$ is the local stretching
of the curvature.

The two dimensional Savage-Hutter equations
\eqref{mass}-\eqref{momentumy} can be collected in a general vector
form as
\begin{equation}
  \frac{\partial \bm{U}}{\partial t} + \frac{\partial \bm{F}(\bm{U})}
  {\partial x} + \frac{\partial \bm{G}(\bm{U})}{\partial y} =
  \bm{S}(\bm{U}), \label{cs}
\end{equation}
where $\bm{U}$ denotes the vector of conservative variables,
$\bm{\mathcal{F}} = (\bm{F}, \bm{G})$ represent the physical fluxes in
the $x$ and $y$ directions, respectively, and $\bm{S}$ is the source
term. They are
\begin{equation}
\begin{aligned}
  \bm{U} = \left[
    \begin{array}{c}
      h \\ hu \\ hv
    \end{array}\right], \quad \quad 
  \bm{F} = \left[
    \begin{array}{c}
      hu \\ hu^2 + \frac{1}{2}\beta_x h^2 \\ huv
    \end{array}\right],  \\
  \bm{G} = \left[
    \begin{array}{c}
      hv \\ huv \\ hv^2 + \frac{1}{2}\beta_y h^2
    \end{array}\right], \quad\quad 
  \bm{S}=\left[
    \begin{array}{c}
      0 \\ hs_x \\ hs_y
    \end{array} \right].
\end{aligned} \label{cs2}
\end{equation}
The flux for \eqref{cs} and \eqref{cs2} along a direction
$\bm{n} = (n_x, n_y)$ is
$\bm{\mathcal{F}} \cdot \bm{n} = \bm{F}n_x +\bm{G}n_y$, while $\bm{n}$
is a unit vector. The Jacobian of the flux along $\bm{n}$ is given by
\[
\bm{J} = \frac{\partial(\bm{F}n_x +\bm{G}n_y )}{\partial \bm{U}} = 
\left[\begin{array}{ccc}
0 & n_x & n_y \\
(\beta_x h-u^2)n_x-uvn_y & 2un_x + vn_y & un_y\\
-uvn_x +(\beta_yh-v^2)n_y  & vn_x & un_x + 2vn_y
\end{array}\right],
\]
where $n_x$ and $n_y$ are the components of the unit vector in the
$x$ and $y$ directions, respectively. The three eigenvalues of the
flux Jacobian are
\[
\begin{aligned}
  \lambda_1 & = u_{\bm{n}} -\sqrt{h(\beta_x n_x^2 + \beta_y 
  n_y^2)},  \\
  \lambda_2 & = u_{\bm{n}}, \\
  \lambda_3 & = u_{\bm{n}} + \sqrt{h(\beta_x n_x^2 + \beta_y 
  	n_y^2)}
\end{aligned}
\]
where $u_{\bm{n}} = \bm{u} \cdot \bm{n}$.

For laboratory avalanches, the ranges of $\delta$ and $\phi$ are
usually made $\beta_x$ and $\beta_y$ greater than zero
\cite{Hutter2005, Pudasaini2007}, thus, the eigenvalues of the flux
Jacobian are all real and are distinct if the avalanche depth $h>0$.
Therefore the Savage-Hutter system given by \eqref{cs} and \eqref{cs2}
is hyperbolic if the gradient of velocity components are not
zeros.

Clearly, the Savage-Hutter equations have a very similar mathematical
structure to the shallow water equations of hydrodynamics at a first
glance. But as a matter of fact, the derivation of Savage-Hutter
equations are not the same as those in the shallow water approximation
although they both start from the incompressible Navier-Stokes
equation. Meanwhile, on account of the jump in the earth pressure
coefficients $K_{x_{act/pass}}$ and $K_{y_{act/pass}}$, the complex
source term $s_x$ and $s_y$, the free surface at the front and rear
margins, it can be quite complicated to develop an appropriate
numerical method to solve the Savage-Hutter equations.

It is a direct observation that the Savage-Hutter equations is not
Galinean invariant. Precisely, it is not rotational invariant if
$\beta_x \neq \beta_y$. It is well-known that the shallow water
equations is rotational invariant. This is an essential difference
between the shallow water equations and the Savage-Hutter
equations. Let us point out this fact by
\begin{theorem}\label{theorem1}
  For any unit vector $\bm{n}=(\cos{\theta}, \sin{\theta})$,
  $\theta \neq 0$, and all vectors $\bm{U}$, the following equality
  \begin{equation}
    \cos\theta \bm{F}(\bm{U}) + \sin\theta \bm{G}(\bm{U}) = 
    \bm{T}^{-1}\bm{F}(\bm{TU}) \label{rotational}
  \end{equation}
  holds if and only if $\beta_x = \beta_y$, where the rotation matrix
  $\bm{T}$ is as
  \[
    \bm{T}=\left[
      \begin{array}{ccc}
        1 & 0 & 0 \\
        0 & \cos\theta & \sin\theta \\
        0 & -\sin\theta  & \cos\theta
      \end{array}\right].
  \]
\end{theorem}
\begin{proof}
First we calculate $\bm{TU}$. The result is 
\[
 \bm{TU}=\left[\begin{array}{c} h \\
 h(u\cos{\theta}+v\sin{\theta})\\
 h(v\cos{\theta}-u\sin{\theta})
 \end{array}\right].
\]
Next we compute $\bm{F}(\bm{TU})$ and obtain
\[
 \bm{F}(\bm{TU})=\left[\begin{array}{c}
 h(u\cos{\theta}+v\sin{\theta}) \\
 \frac{h^2\beta_x}{2}+\frac{(hu\cos{\theta}+hv\sin{\theta})^2}{h}\\
 \frac{(hv\cos{\theta}-hu\sin{\theta})(hu\cos{\theta}+hv\sin{\theta})}{h}
 \end{array}\right].
\]
Then we apply the inverse rotation $\bm{T}^{-1}$ to $\bm{F}(\bm{TU})$
and get
\begin{equation}
\bm{T^{-1}F}(\bm{TU})=\left[\begin{array}{c}
h(u\cos{\theta}+v\sin{\theta}) \\
\frac{1}{2}h(2u^2+\beta_x)\cos{\theta}+huv\sin{\theta} \\
huv\cos{\theta}+\frac{1}{2}h(2v^2+h\beta_x)\sin{\theta}
\end{array}\right]. \label{RH}
\end{equation}
For the left hand of Eq. \eqref{rotational}, we have 
\begin{equation}
\cos\theta \bm{F}(\bm{U}) + \sin\theta \bm{G}(\bm{U}) = \left[ 
\begin{array}{c}
h(u\cos{\theta}+v\sin{\theta}) \\
\frac{1}{2}h(2u^2+\beta_x)\cos{\theta}+huv\sin{\theta} \\
huv\cos{\theta}+\frac{1}{2}h(2v^2+h\beta_y)\sin{\theta}
\end{array}\right].  \label{LH}
\end{equation}
Thus, if $\beta_x \ne \beta_y$, \eqref{RH} is not equal to
\eqref{LH}.	
\end{proof}

If the system is rotational invariant, one may solve a 1D Riemann
problem with $(h, h u_{\bm{n}})$, where
$u_{\bm{n}} = \bm{u} \cdot \bm{n}$, as variable and get a 1D numerical
flux $(f_h, f_m)$, then the numerical flux along $\bm{n}$ for 2D
problem is given as $(f_h, f_m \bm{n})$. This procedure is very
popular while it only works in case that
$\bm{\mathcal{F}} \cdot \bm{n} = \bm{T}^{-1} \bm{F}(\bm{TU})$
holds. The theorem tells us that this equality holds only when
$\theta = 0$ or $\beta_x = \beta_y$. In a finite volume method to be
investigated, $\bm{n}$ will be set as the unit normal of the cell
interfaces, that it can not always vanish $\theta$ in unstructured
grids. Therefore, if $\beta_x \neq \beta_y$, any approximate Riemann
solver based on such a procedure is out of the scope of our
candidate numerical flux.



\section{Numerical Method}
Let us discretize the Savage-Hutter equations \eqref{cs} and
\eqref{cs2} on the Delaunay triangle grids using a finite volume
Godunov-type approach. We will adopt the modified HLL approximate
Riemann Solver to calculate numerical fluxes across cell
interfaces. The MUSCL-Hancock method and ENO-type reconstruction
technique are adopted to achieve second-order accuracy in space and
time.

\subsection{Finite volume method}
Before introducing the cell-centered finite volume method, the entire
spatial domain $\Omega$ is subdivided into $N$ triangular cells
$\tau_i, i = 1, 2, \cdots N$. In each cell $\tau_i$, the conservative
variables of Savage-Hutter equations are discretized as piecewise
linear functions as
\[
	\left.\bm{U}^{n}\right|_{\tau_{i}}(\bm{x})=\bm{U}_{i}^{n}+
	\nabla \bm{U}_{i}^{n}\left(\bm{x}-\bm{x}_{i}\right)
\]
where $\bm{U}_{i}^{n}=\left(h_{i}^{n},(h u)_{i}^{n},(h
v)_{i}^{n}\right)^{T}$ is the cell average value, $\nabla
\bm{U}_{i}^{n}$ is the slope on $\tau_i$ and $\bm{x}_i$ is the
barycenter of $\tau_i$, the superscript $n$ indicates at time $t^n$.

To introduction the finite volume scheme, Eq. \eqref{cs} is 
integrated in a triangle cell $\tau_i$ 
\begin{equation}
	\frac{\partial}{\partial t}\int_{\tau_i}{\bm{U}}d\Omega +
	 \int_{\tau_i}\left(\frac{\partial \bm{F}}{\partial x} +
	 	 \frac{\partial \bm{G}}{\partial y} \right) d\Omega 
	 	= \int_{\tau_i}{\bm{S}}d\Omega.  \label{intcs}
\end{equation}
Applying Green's formulation, Eq.\eqref{intcs} becomes
\[
	\frac{\partial}{\partial t}\int_{\tau_i}{\bm{U}}d\Omega +
	\int_{\partial\tau_i}{\bm{\mathcal{F}} \cdot\bm{n}}ds 
	=\int_{\tau_i}{\bm{S}}d\Omega,
\]
where $\partial\tau_i$ denotes the boundary of the cell $\tau_i$;
$\bm{n}$ is the unit outward vector normal to the boundary; $ds$ is
the arc elements. The integrand $\bm{\mathcal{F}}\cdot\bm{n}$ is the
outward normal flux vector in which $\bm{\mathcal{F}} = \left[\bm{F},
\bm{G}\right]$. Define the cell average:
\[
	\bm{U}_i=\frac{1}{|\tau_i|}\int_{\tau_i}{\bm{U}}d\Omega,
\]
where $|\tau_i|$ is the area of the cell $\tau_i$.
The method achieves second-order accuracy by using the MUSCL-Hancock
method which consists of predictor and corrector steps. The procedure
of the predictor-corrector Godunov-type method goes as follows:

\textbf{Step 1: Predictor step}
\[
	\bm{U}_{i}^{n+1/2} = \bm{U}_i - \frac{\Delta 
	t}{2|\tau_i|}\sum_{j=1}^{3}{\int_{\partial\tau_i,j}{\bm{\mathcal{F}}(\bm{U}_{in}^{n})\cdot
	 \bm{n}_{j}}}ds + \frac{\Delta t}{2}\bm{S}_{i}^{n},
\]
where $\partial \tau_{i,j}$ is the $j$-th edge of $\tau_{i}$ with the
unit outer normal as $\bm{n}_{j}$, and $\bm{S}_{i}$ is source term
discretized by a centered scheme. The flux vector
$\bm{\mathcal{F}}(\bm{U}_{in}^{n})$ is calculated at each cell face
$\partial \tau_{i,j}$ after ENO-type piecewise linear reconstruction.
$\Delta t$ is the time step.

\textbf{Step 2: Corrector step}
\[
	\bm{U}_{i}^{n+1} = \bm{U}_{i}^{n} - \frac{\Delta 
	t}{|\tau_i|}\sum_{j=1}^{3}{\int_{\partial 
	\tau_i,j}{\bm{\mathcal{F}}_{j}^{*}(\bm{U}_{in}^{n+1/2},\bm{U}_{out}^{n+1/2})\cdot\bm{n}_{j}ds}
	 } + \Delta t \bm{S}_{i}^{n+1/2}.
\]
where the numerical flux vector
$\bm{\mathcal{F}}^{*}(\bm{U}_{in}^{n+1/2},\bm{U}_{out}^{n+1/2})\cdot\bm{n}$
is evaluated based on an approximate Riemann solver at each
quadrature point on the cell boundary. In the paper, we use HLL
approximate Riemann solver owe to its simplicity and stable
performance on wet/day interface.

\subsection{HLL approximate Riemann solver}
For a lot of hyperbolic systems, it uses the approximate Riemann
solvers to obtain the approximate solutions in most practical
situations. Since Savage-Hutter equations and shallow water equations
are very similar, we can directly extend approximate Riemann solvers 
of shallow water equations to Savage-Hutter equations. The HLL 
approximate Riemann solver can offer a simple way of dealing with dry 
bed situaitons and the determination of the wet/dry front velocities 
\cite{Fraccarollo2010}. Similar to shallow water equations, the 
numerical interface flux of HLL for Savage-Hutter equations is 
computed as follows
\[
	\bm{\mathcal{F}}^{*}\cdot \bm{n} = \left\{ \begin{array}{lc}
	\bm{\mathcal{F}}(\bm{U}_L)\cdot\bm{n},  & \mathrm{if} \quad s_L \ge 0, \\
	\dfrac{s_R \bm{\mathcal{F}}(\bm{U}_L)\cdot \bm{n} - s_L 
	\bm{\mathcal{F}}(\bm{U}_R)\cdot \bm{n} + s_R s_L(\bm{U}_R - 
	\bm{U}_L)}{s_R - s_L}, & \mathrm{if} \quad s_L \le 0 \le s_R, \\
	\bm{\mathcal{F}}(\bm{U}_R)\cdot\bm{n},  &  \mathrm{if} \quad s_R \le 0,
	\end{array}
	\right.
\]
where $\bm{U}_L = (h_L, (hu)_L, (hv)_L)$ and $\bm{U}_R = (h_R, 
(hu)_R, (hv)_R)$ are the left and right Riemann states for a local 
Riemann problem, respectively. $ s_L$ and $s_R$ are estimate of the 
speeds of the left and right waves, respectively. The key to the 
success of the HLL approach is the availiability of estimates for the 
wave speeds $s_L$ and $s_R$. For the Savage-Hutter equations, there 
are also dry/wet bed cases. Thus, in order to take the dry bed into 
account,  similar to shallow water equations, the left and right wave 
speeds are expressed as
\[
	s_L = \left \{ \begin{array}{lc}
	\min(\bm{u}_L\cdot \bm{n} - \sqrt{c_L h_L}, 
	u_{*} - \sqrt{c_L h_{*}}), & \mathrm{if} \quad h_L, h_R > 0, 
	\\
	\bm{u}_{L}\cdot\bm{n} - \sqrt{c_L h_L}, & \mathrm{if}
	\quad 
	h_R = 0, \\
	\bm{u}_{R}\cdot\bm{n} - 2\sqrt{c_R h_R}, & \mathrm{if}
	\quad h_L = 0,
	\end{array}
	\right. 
\]

\[
	s_R =  \left \{ \begin{array}{lc}
	\max(\bm{u}_R\cdot \bm{n} + \sqrt{c_R h_R}, 
	u_{*} + \sqrt{c_R h_{*}}), & \mathrm{if} \quad h_L, h_R > 0, 
	\\ [2mm]
	\bm{u}_{L}\cdot\bm{n} + 2\sqrt{c_L h_L}, & \mathrm{if} 
	\quad h_R = 0, \\ [2mm]
	\bm{u}_{R}\cdot\bm{n} + \sqrt{c_R h_R}, & \mathrm{if}
	\quad 
	h_L = 0,
	\end{array}
	\right.
\]
where $\bm{u}_L = (u_L, v_L)^{T}$, $\bm{u}_R = (u_R, v_R)^{T}$, 
$c_L = (\beta_x n_x^2 + \beta_y n_y^2)_L$, $c_R = (\beta_x n_x^2 + 
\beta_y n_y^2)_R$ and 
\[
	\left\{ \begin{array}{lc}
	u_{*} = \frac{1}{2}(\bm{u}_L + \bm{u}_R)\cdot \bm{n} + 
	\sqrt{c_L h_L} - \sqrt{c_R h_R} \\ [2mm]
	h_{*} = 
	\frac{1}{\bm{\beta}\cdot\bm{n}}\left[\frac{1}{2}(\sqrt{c_L h_L}
	 + \sqrt{c_R h_R}) + \frac{1}{4}(\bm{u}_L - 
	\bm{u}_R)\cdot\bm{n}\right]^{2}.
	\end{array}
	\right.
\]

\subsection{Linear reconstruction}
In order to get second order accuracy in spatial, a simple ENO-type 
piecewise linear reconstruction was applied to suppress the numerical 
oscillations near steep gradients or discontinuites. This 
reconstrction method had been used by Deng et. al. for shallow water 
equations \cite{Deng2013}.

\begin{figure}[htbp] 
		\centering 
		\includegraphics[width=0.3\textwidth]{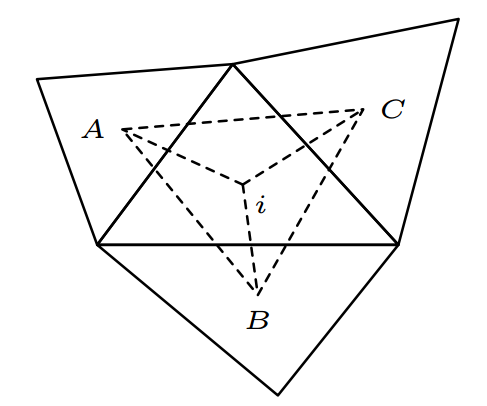} 
		\caption{A schematic of the ENO-type reconstruction on patch of cells (reproduced from \cite{Deng2013}).}
		\label{fig:ENO} 	
\end{figure}

This part is mainly taken from \cite{Deng2013}.
For a variable $\psi=h, hu$ or $hv$, assuming its cell average value 
$\psi_i$ on $\tau_i$ is given (see Fig. \ref{fig:ENO}), then we can 
get the gradients of 
$\psi$ for the three triangles $\triangle AiB$, $\triangle BiC$ and 
$\triangle CiA$ are referred to as $(\nabla \psi)_{AiB}$, $(\nabla 
\psi)_{BiC}$ and $(\nabla \psi)_{CiA}$, and choose the one with the 
minimal $l^{2}$ norm as $\nabla \psi_i$, that is, 
\[
	\nabla \psi_{i}=\underset{\nabla 
	\psi}{\operatorname{min}}\left\{\|\nabla \psi\|_{l^{2}}, \nabla 
	\psi \in\left\{(\nabla \psi)_{A i B},(\nabla \psi)_{B i 
	C},(\nabla \psi)_{C i A}\right\}\right\}
\]

Acoording to the idea of minmod slope limiter, we let $\nabla \psi_i 
= 0$ when $\psi_i \geq \max{\{\psi_A, \psi_B, \psi_C\}}$ or $\psi_i 
\leq \min{\{ \psi_A, \psi_B, \psi_C\}} $. Meanwhile, in the process 
of reconstructing $h$, we must consider the physical criteria, that 
is, the avalanche depth $h$ must be non-negative at each quadrature 
points, otherwise $\nabla h_i$ is set to be zero.

\subsection{Time stepping}
The time-marching formula is explicit, and the time step length 
$\Delta t$ is determined by the Courant-Friedrichs-Lewy(CFL) 
condition for its stability. For triangular grid and implemention of 
CFL condition, the time step length is usually expressed as
\[
	\Delta t=\mathrm{C_r} \cdot \min _{i} \frac{\Delta 
	x_{i}}{\sqrt{|\bm{\beta}| h_{i}}+\sqrt{u_{i}^{2}+v_{i}^{2}}}, 
	\quad i = 1, 2, \cdots, N
\]
where $\mathrm{C_r}$ is the Courant number specified in the range 
$0<\mathrm{C_r} \leq 1$, and $\mathrm{C_r} = 0.5$ is adopted in our 
simulations; $|\bm{\beta}| = \sqrt{\beta_x^2 + \beta_y^2}$, $N$ is 
the total number of cells, and $\Delta x_i$ is the size of $\tau_i$. 
For triangular grid, $\Delta x_i$ is usually calculated as the 
minimum barycenter-to-barycenter distance between $\tau_i$ and its 
adjacent cells. 



\section{Numerical Examples}

In this section, three numerical examples were performed to verify 
the proposed 
scheme which implemented using C++ programming language based on the 
adaptive finite element package {\tt AFEPack}\cite{Li2006}. The initial 
triangular grids were generated by easymesh. Because the posteriori 
error estimator for $h$-adaptive method on unstructured grid was not 
the main goal of this work, we just applied the following local error 
indicator which advised by \cite{Deng2013}
\[
	\mathcal{E}_{\tau}=\mathcal{E}_{\tau}(h)+\mathcal{E}_{\tau}(h 
	u)+\mathcal{E}_{\tau}(h v)
\]
where
\[
	\mathcal{E}_{\tau}(\psi)=|\tau| \int_{\partial \tau} \left( 
	\frac{1}{\sqrt{|\tau|}}\left|\left[[I_{h} \psi]\right]\right| + 
	\left| \left[[ \nabla I_{h} \psi ]\right] \right| \right) 
	\mathrm{d\mathbf{l}}.
\]
in which $\tau$ is a cell in the mesh and $|\tau|$ is its area, 
$\left[[ \cdot \right]] $ denotes the jump of a variable across 
$\partial \tau$, and $I_h \psi$ is the piecewise linear numerical 
solution for $\psi = h, hu, hv$, respectively. 
\subsection{Dam break test cases with exact solution}
For this dam break problem, it was originally a one dimensional 
problem with analytical solutions involving a constant bed slope and 
a Coulomb frictional stress, thus, its governing equations can 
reduced as \cite{Juez2013}\cite{Zhai2015}
\[
	\left\{ \begin{array}{l}
	{\dfrac{\partial h}{\partial t} + \dfrac{\partial (hu)}{\partial x} 
	= 0 }  \\ [2mm]
	{\dfrac{(hu)}{\partial t} + \dfrac{(hu^2 + \frac{1}{2} g\cos{\zeta} 
	h^2)}{\partial x} = -gh\cos{\zeta}(\tan{\delta}-\tan{\zeta})}
	\end{array}
	\right.
\]
where $g$ is the gravitational constant, the analytical solution of a 
granular dam break problem is given  
\cite{Zhai2015}:
\[
	(h, \mathcal{U})=\left\{\begin{array}{ll}{\left(h_{0}, 0\right),} 
	& {\chi< -c_{0} t} \\ 
	{\left(\frac{h_{0}}{9}\left(2-\frac{\chi}{c_{0} t}\right)^{2}, 
	\frac{2}{3}\left(\frac{\chi}{t}+c_{0}\right)\right),} & {-c_{0} t 
	\leq \chi \leq 2 c_{0} t} \\ {(0,0),} & {\chi>2 c_{0} 
	t}\end{array}\right.
\]
where $c_0 = \sqrt{g h_0 \cos{\zeta}}$ and 
\[
	\chi=x+\frac{1}{2} g \cos \zeta(\tan \delta-\tan \zeta) t^{2},
\]
\[
	\mathcal{U}=u+g \cos \zeta(\tan \delta-\tan \zeta) t.
\]
In order to check the performance of the numerical schemes proposed 
in this work, we modified it to an equivalent two-dimensional problem 
as Ref.\cite{Zhai2015}. The computational domain is $[-12.8,12.8] 
\times[-1.6,1.6]$, the initial conditions are
\[
	\left\{\begin{array}{ll}{(h, u, v)=(10,0,0),} & {x<0} \\ 
	{(h, u, v)=(0,0,0),} & {x>0}\end{array}\right.
\]
and $\zeta=40^{\circ}, \delta = 24.5^{\circ} $.

Fig. \ref{fig::1D} plots the compute avalanche depth in comparison 
with the exact solutions at $t=0.1, 0.2, 0.3, 0.4$ and $0.5$. It can 
be seen that our numerical solutions are in agreement with exact 
solutions very well, which also verifies the robustness and validity 
of our numerical algorithm.

\begin{figure}[htbp] 
	\begin{minipage}[t]{0.45\linewidth} 
		\centering 
		\includegraphics[width=1.2\textwidth]{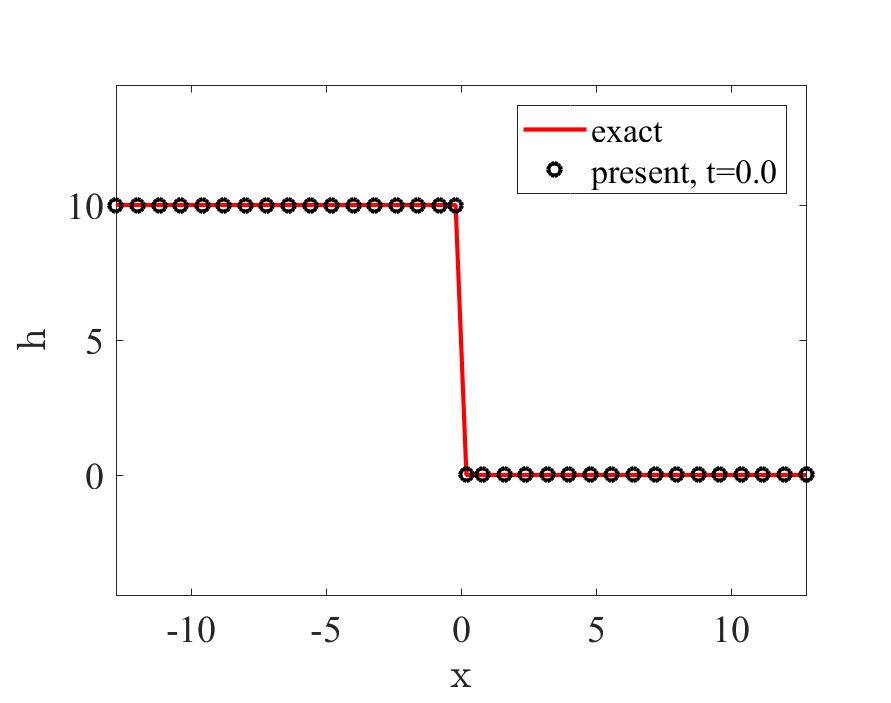} 
	\end{minipage}
	\hfill
	\begin{minipage}[t]{0.45\linewidth}
		\centering 
		\includegraphics[width=1.2\textwidth]{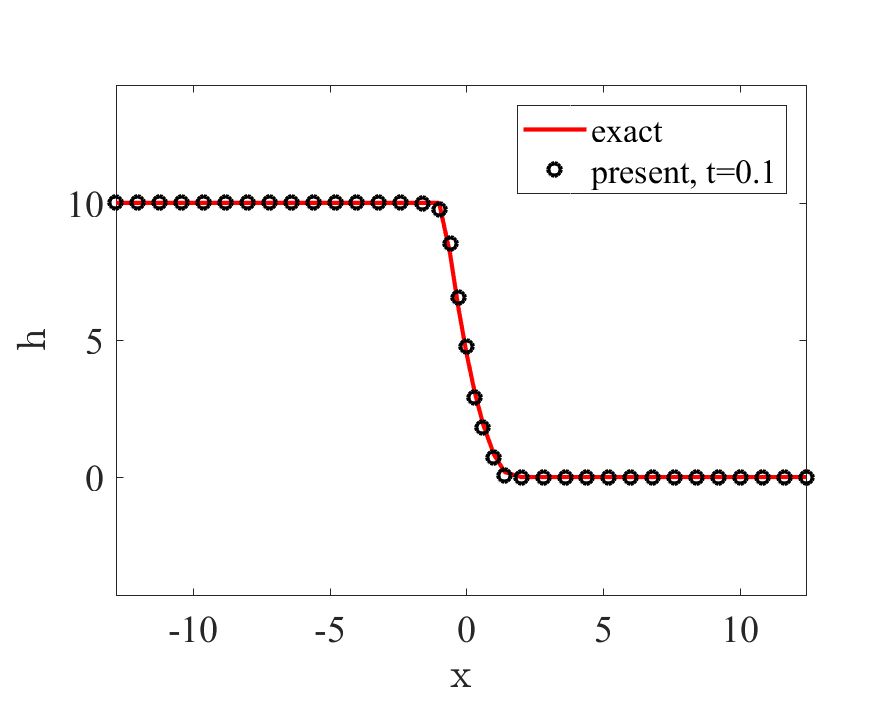}
		 \end{minipage} 
		
		 \begin{minipage}[t]{0.45\linewidth} 
		 \centering 
		 \includegraphics[width=1.2\textwidth]{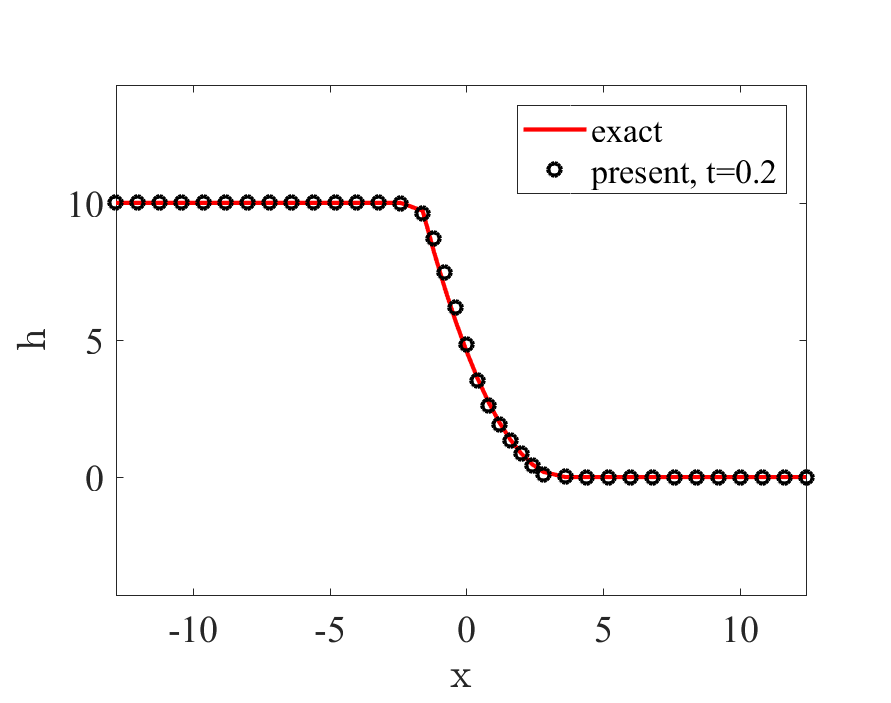} 
		\end{minipage}
		 \hfill  
		\begin{minipage}[t]{0.45\linewidth} 
			\centering 
			\includegraphics[width=1.2\textwidth]{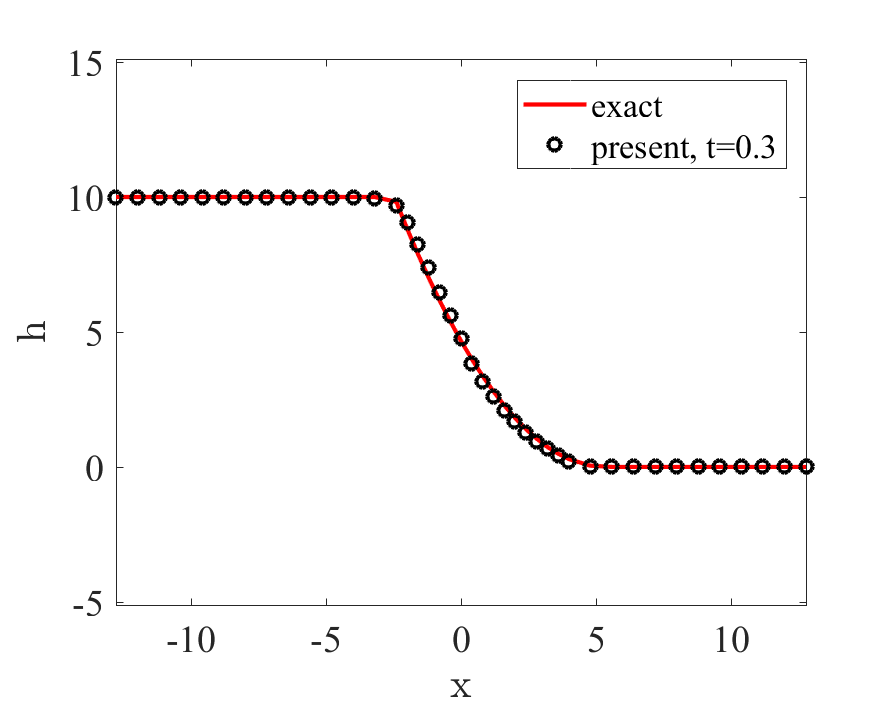} 
		\end{minipage}
		\begin{minipage}[t]{0.45\linewidth} 
			\centering 
			\includegraphics[width=1.2\textwidth]{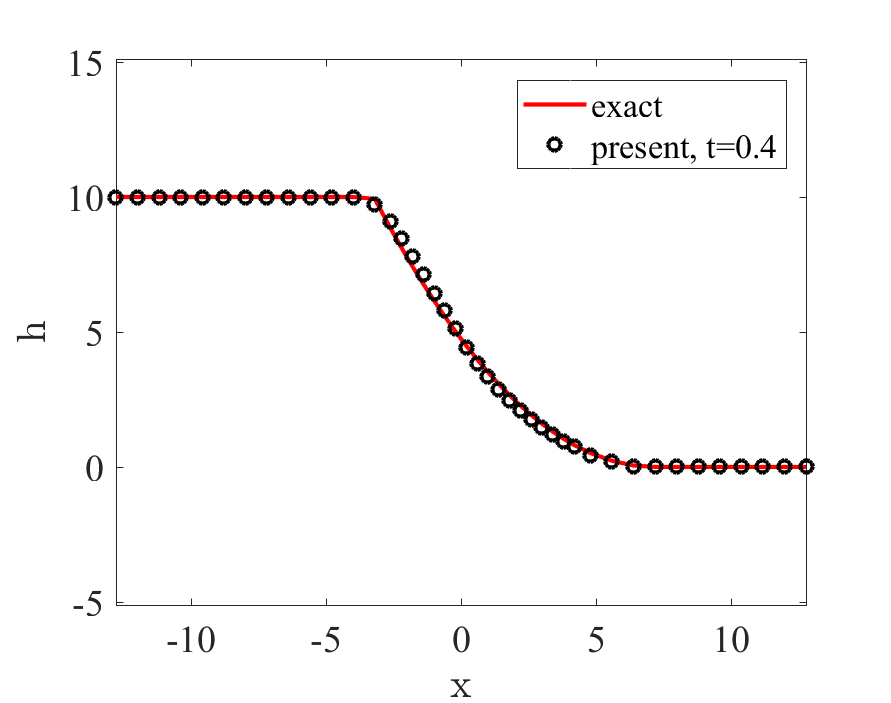} 
		\end{minipage}
		 \hfill 
		\begin{minipage}[t]{0.45\linewidth} 
			\centering 
			\includegraphics[width=1.2\textwidth]{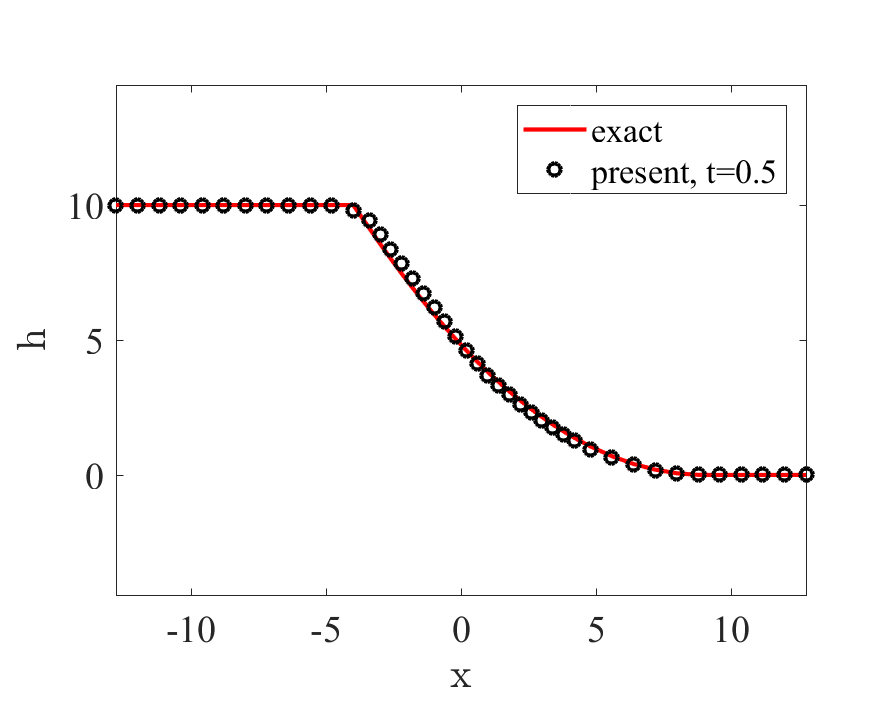} 
		\end{minipage}
		
		\centering
		\caption{Exact and numerical solutions for the 1D dam break problem at $t=0.0,0.1,0.2,0.3,0.4$ and $0.5$ }
		\label{fig::1D}
\end{figure}

\subsection{Avalanche flows slide down an inclined plane and merge 
	continuously into a horizontal plane}
In this section, we simulate an avalanche of finite granular mass 
sliding down an inclined plane and merging contiuously into a 
horizontal plane, which was considered in \cite{Wang2004} to verify 
their model's ability.

In order to compare with the results from Ref. 
\cite{Wang2004}\cite{Zhai2015}, all computational parameters are set 
to the same as they are. That is,
the computational domain is $[0,30]\times[-7, 7]$, and $\epsilon = 
\lambda = 1$ in our simulations. The inclined section lies $x \in [0, 
17.5]$ and the horizontal region lies $x \in [21.5, 30]$ with a 
smooth transition zone lies $x \in [17.5, 21.5]$. The inclination 
angle is given by
\[
	\zeta(x)=\left\{\begin{array}{ll}{\zeta_{0},} & {0 \leq x \leq 
	17.5} \\ {\zeta_{0}\left(1-\frac{x-17.5}{4}\right),} & 
	{17.5<x<21.5} \\ {0^{\circ},} & {x \geq 21.5} \end{array}\right.
\]
where $\zeta_{0} = 35^{\circ}$ and $\delta = \phi = 30^{\circ}$. The 
granular mass is suddenly released at $t=0$ from the hemispherical 
shell with an initial radius of $r_0 = 1.85$ in dimensionless length 
units. The center of the cap is initially located at $(x_0, y_0)=(4, 
0)$.

Fig. \ref{fig::ex2} shows the avalanche depth contours of the fluid 
at eight time slices as the avalanche slides on the inclined plane 
into the horizontal run-out zone. From the Fig. \ref{fig::ex2} it 
can be seen that, the avalanche starts to flow along the $x$ and $y$ 
direction due to gravity, and it flows faster along the downslope 
direction obviously. At $ t= 9$ the part of granular reaches the 
horizontal run-out zone and begins to deposite because of the basal 
friction is sufficient large, and  the tail is still doing the 
acceleration movement. Here, some small numerical oscillations are 
visible in the solutions which also found in 
Ref.\cite{Wang2004}\cite{Zhai2015} with orthogonal quadrilateral 
mesh.  At $ t= 12$ a shock wave 
develops round the end of the transition zone at $ x = 21.5$ and a 
surge wave is created. During the $t=12$ to $t=24$, most of graunlar 
mass are accumulated at the end of the transition zone and the 
front-end of the horizontal zone, meanwhile, a obvious backward surge 
can be found. At $t=24$ the final deposition of the avalanche is 
nearly formed and an expansion fan is observed.

On the whole, our numerical resluts are identical with these 
available from Ref. \cite{Wang2004}\cite{Zhai2015}. Fig. 
\ref{fig::ex2mesh} presents the mesh adaptive moving for the 
avalanche granular flows down the inclined plane and merges 
contiuously into a horizontal plane at $ t=6, 12, 18, 24$. It can be 
seen that the triangle mesh is refined at the free surface where 
gradients varied significantly.     

\begin{figure}[htbp] 
	\begin{minipage}[t]{0.45\linewidth}
		\centering 
		\includegraphics[width=1.2\textwidth]{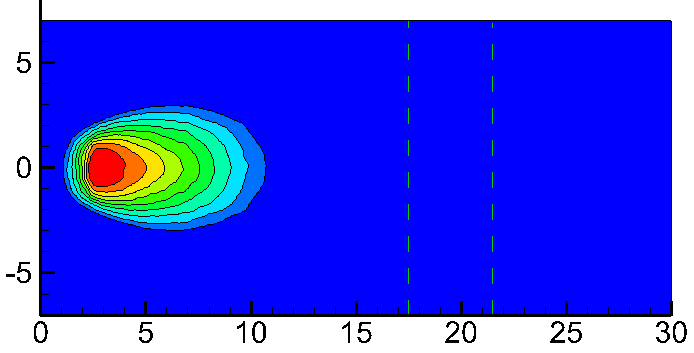}
	\end{minipage} 
	\hfill 
	\begin{minipage}[t]{0.45\linewidth} 
		\centering 
		\includegraphics[width=1.2\textwidth]{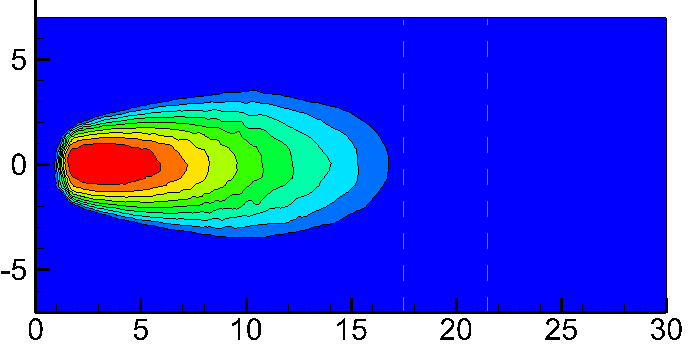} 
	\end{minipage} 
	\begin{minipage}[t]{0.45\linewidth} 
		\centering 
		\includegraphics[width=1.2\textwidth]{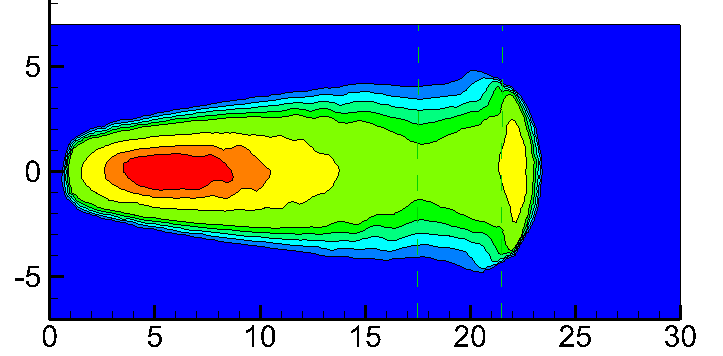} 
	\end{minipage}
	\hfill
	\begin{minipage}[t]{0.45\linewidth} 
		\centering 
		\includegraphics[width=1.2\textwidth]{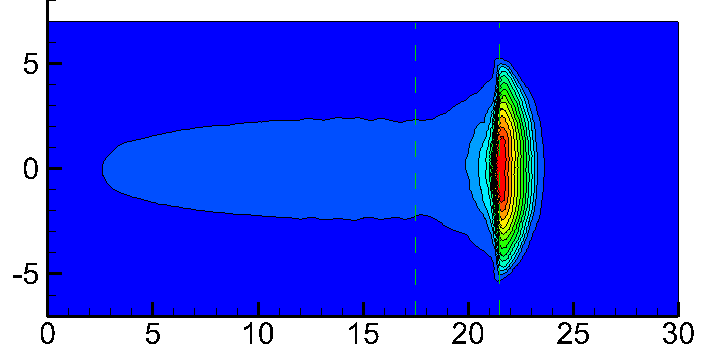} 
	\end{minipage}
	\begin{minipage}[t]{0.45\linewidth} 
		\centering 
		\includegraphics[width=1.2\textwidth]{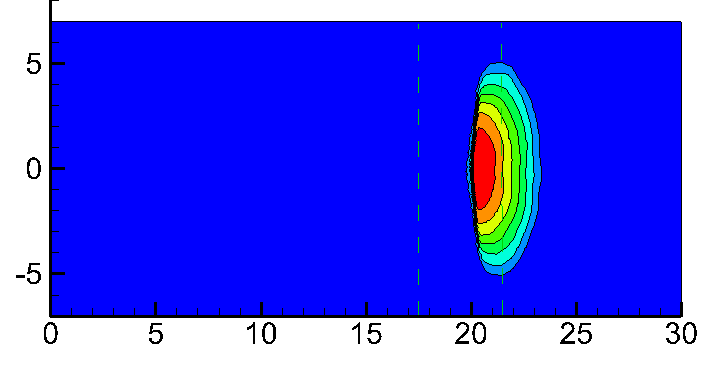} 
	\end{minipage}
	\hfill
	\begin{minipage}[t]{0.45\linewidth} 
		\centering 
		\includegraphics[width=1.2\textwidth]{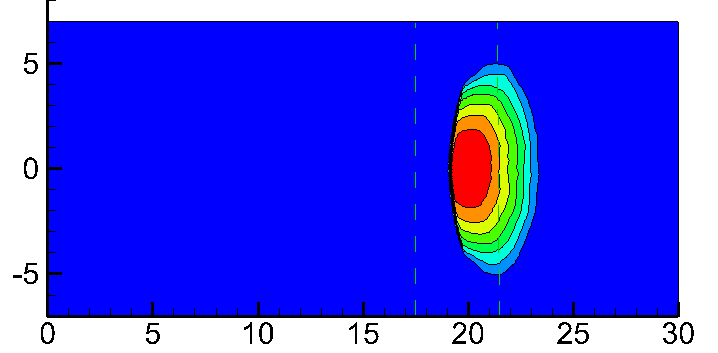} 
	\end{minipage}
	\begin{minipage}[t]{0.45\linewidth} 
		\centering 
		\includegraphics[width=1.2\textwidth]{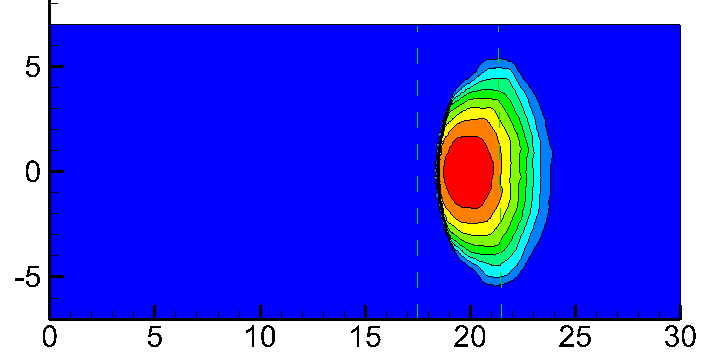} 
	\end{minipage}
	\hfill
	\begin{minipage}[t]{0.45\linewidth} 
		\centering 
		\includegraphics[width=1.2\textwidth]{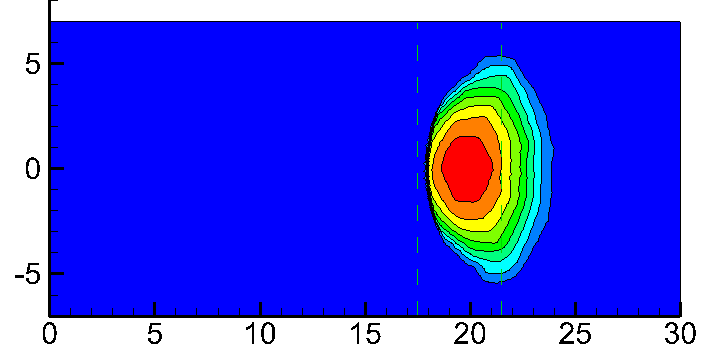} 
	\end{minipage}
	\centering
	\caption{Thickness contours of the avalanche at eight different 
	dimensionless times $t=3, 6, 9, 12, 15, 18, 21, 24$ for the flow 
	slides down the inclined plane and merging continuously into a 
	horizontal plane. The transition zone from the inclined plane to 
	the horizontal plane lies between the two green dashed lines.  }
	\label{fig::ex2}
\end{figure}

\begin{figure}
	\begin{minipage}[t]{0.45\linewidth}
		\centering 
		\includegraphics[width=1.2\textwidth,trim=80 50 80 430,clip]{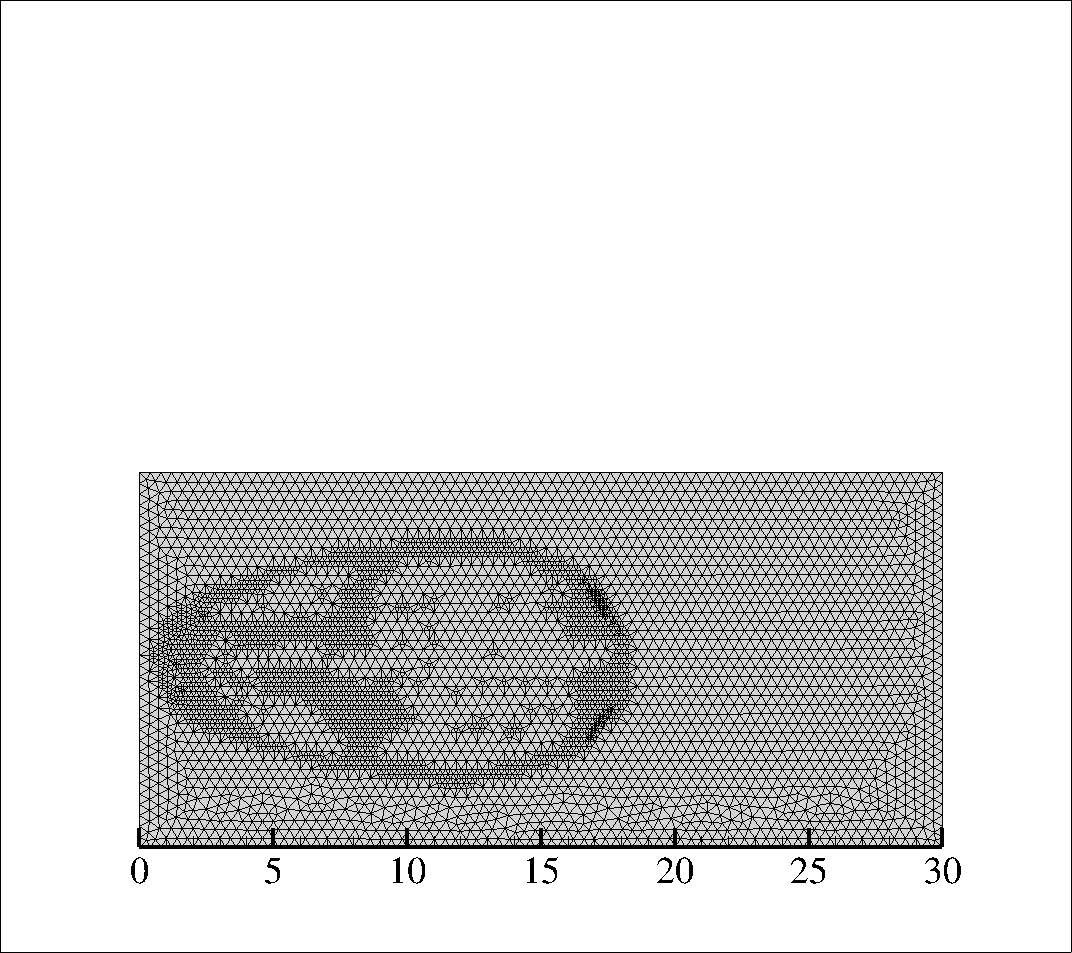}
	\end{minipage} 
	\hfill 
	\begin{minipage}[t]{0.45\linewidth} 
		\centering 
		\includegraphics[width=1.2\textwidth,trim=80 50 80 430,clip]{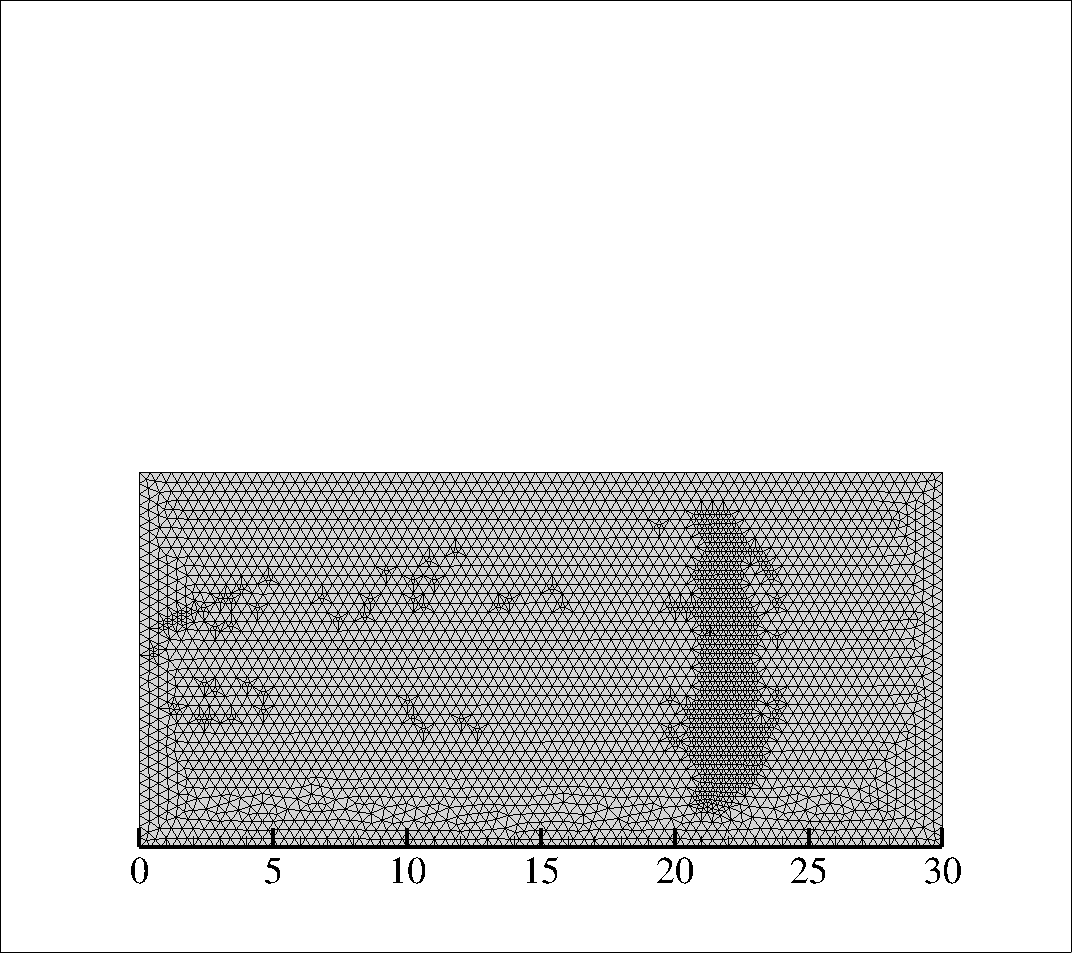}
	\end{minipage} 
	\begin{minipage}[t]{0.45\linewidth} 
		\centering 
		\includegraphics[width=1.2\textwidth,trim=80 50 80 430,clip]{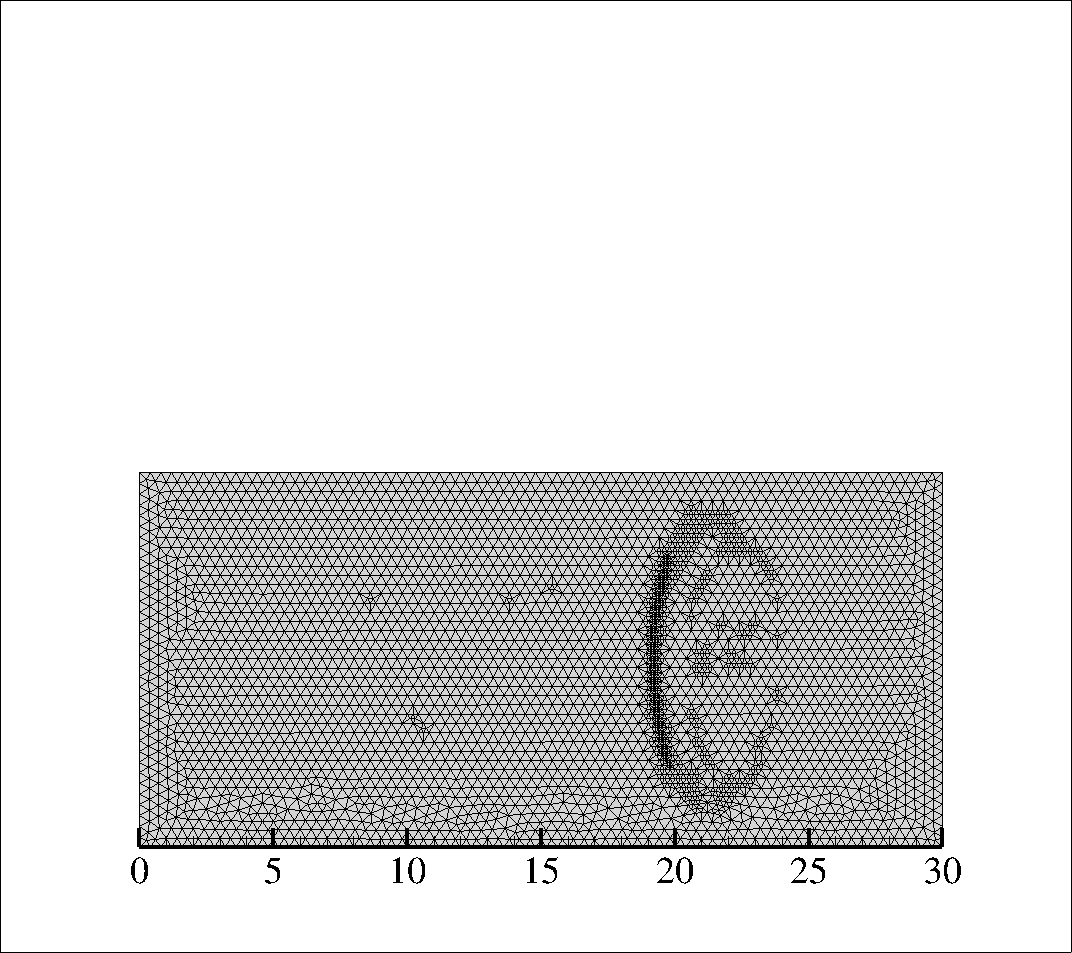} 
	\end{minipage}
	\hfill
	\begin{minipage}[t]{0.45\linewidth} 
		\centering 
		\includegraphics[width=1.2\textwidth,trim=80 50 80 430,clip]{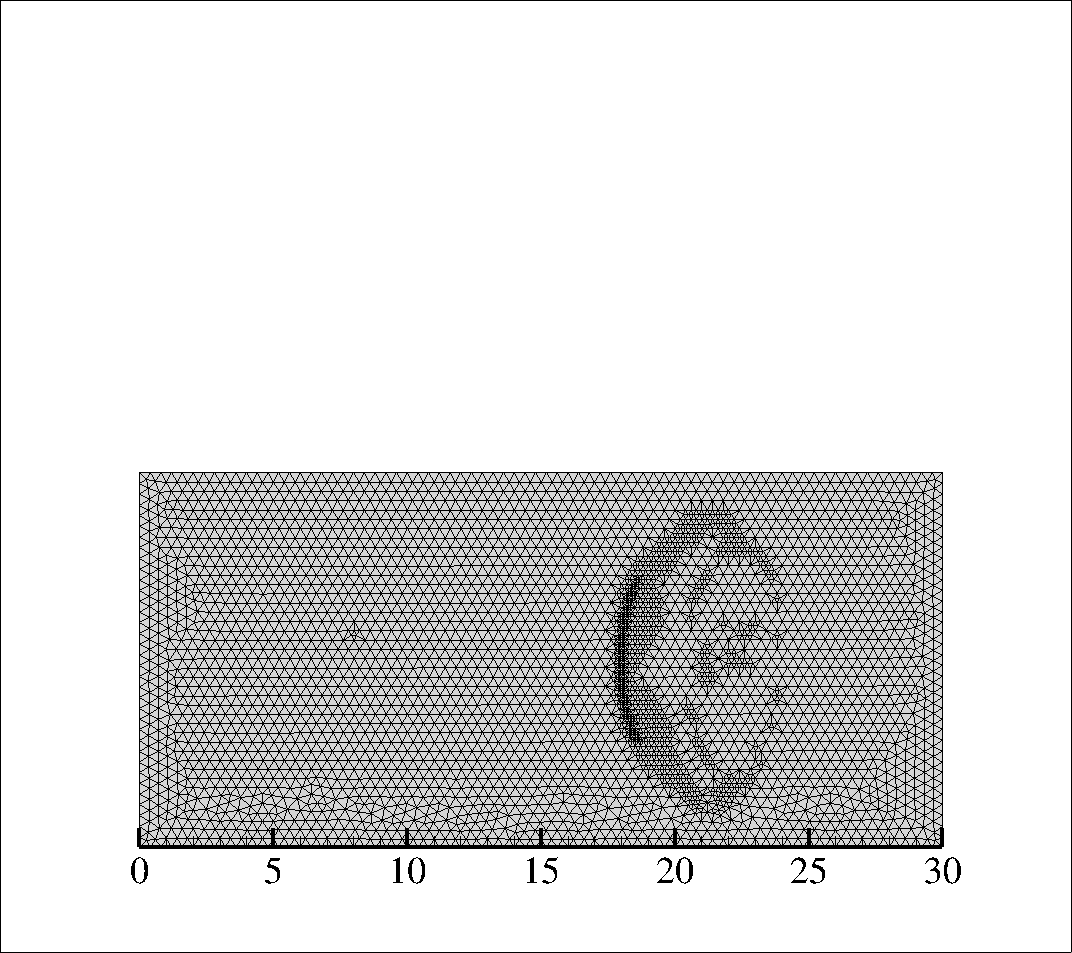}
	\end{minipage}
	\centering
	\caption{The triangular meshes at times $t=6, 12, 18, 24$ for the 
	flow slides down the inclined plane and merging continuously into 
	a horizontal plane.}
	\label{fig::ex2mesh}	
\end{figure}

\subsection{Granular flow pasts obstacle}
The research of granular avalanches around obstacle is of partiacular 
interest because some important pheomena such as shock waves, 
expansion fans and granular vacuum can be observed \cite{Cui2014}. 
Meanwhile, setting obstacles is an effective means to control the 
dynamic process of granular avalanche flow.  In this section, we 
present a simulation example of finite granular mass flows down an 
inclined plane and merges inito a horizontal run-out zone with a 
conical obstacle located on the inclined plane.  
In addition, in order to compared the effects of with and without 
obstacle, all compuational parameters (such as computational domain, 
the inclination angle $\zeta(x)$, $\phi$, $\delta$ ) are the same as 
the above example. A conical obstacle is incorporated into this model 
and the centre of obstacle is located at $(13,0)$ with dimensionless 
radius of bottom-circle of 1 and dimsionless heights $H=1$.

Fig. \ref{fig::ex3} shows a series of snap-shots of the thickness 
contours as the avalanche slides down the inclined plane which has a 
conical obstacle in the avalanche track.
It can be seen that the flow of granular is the same as that without 
obstacle before it reaches the obstacle (eg. for $t \leq 3$). At 
$t=6$ the granular avalanche has reache the obstacle, some of them 
climb the obstacle, and the other divide into two parts and flow 
around the obstacle downwards. At the moment, a stationary zone 
appears in the front of the obstacle, where the avalance velocity 
becomes zero, but it is followed by a great increase in the avalanche 
thickness. Meanwhile, behind the obstacle, a so-called granular 
vacuum is formed which can protect the zone directly behind it from 
disasters.   
At $t=9$, the granular mass continues to accelerate until it reaches 
the horizontal run-out zone where the basal friction is sufficient 
large, which results in the front becomes to rest, but the part of 
the tail accelerates further. At $t=12$ the shcoks waves fromed on 
the horizontal run-out zone during the front of granular mass 
deposits and the tail of it still slides accelerated.
From $t=15$ to $t=18$, most of the granualr mass has slided down 
either side of the obstacle and on the transition zone two expansion 
fans gradually formed. For $ t= 21, 24$ two separate and symmetrical 
deposition mass of granular are formed at the end of transition zone 
and the front of horizontal zone. From $t = 6$ to $ t = 24$ there 
always has  the flow-free region directly behind the obstacle, which 
can be regarded as the protected area in the practial application of 
avalanche protection.  Thus, the obstacle does change the path of 
granular avalanche flow movement and affects its final deposition 
pattern. The numerical results show that our numerical model has the 
ability of capturing key qualitative features, such as shocks wave 
and flow-free region. As in the previous example, Fig. 
\ref{fig::ex3mesh} also shows the adaptive mesh moving for the 
granular avalanche flows down with a obstacle at $ t=6, 12, 18, 24$. 
\begin{figure}[htbp] 
	\begin{minipage}[t]{0.45\linewidth}
		\centering 
		\includegraphics[width=1.2\textwidth]{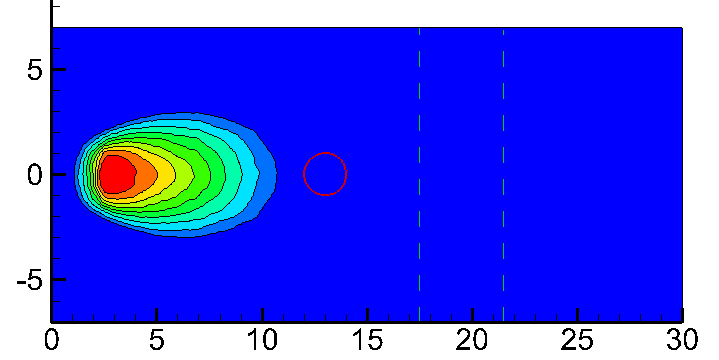} 
	\end{minipage} 
	\hfill 
	\begin{minipage}[t]{0.45\linewidth} 
		\centering 
		\includegraphics[width=1.2\textwidth]{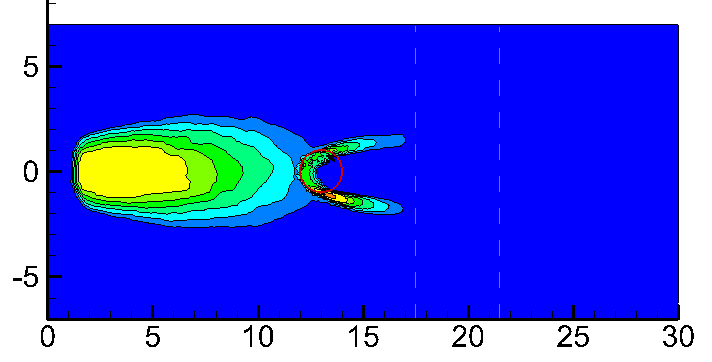} 
	\end{minipage} 
	\begin{minipage}[t]{0.45\linewidth} 
		\centering 
		\includegraphics[width=1.2\textwidth]{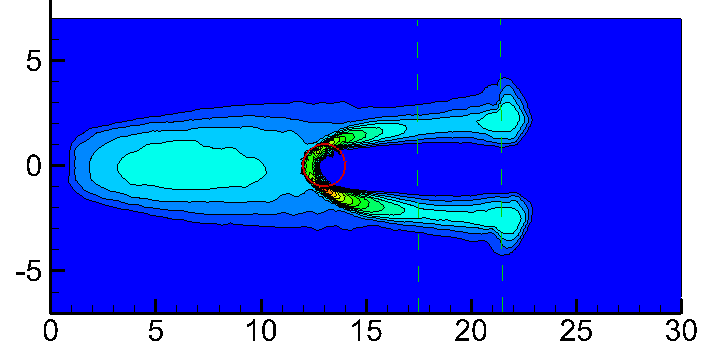} 
	\end{minipage}
	\hfill
	\begin{minipage}[t]{0.45\linewidth} 
		\centering 
		\includegraphics[width=1.2\textwidth]{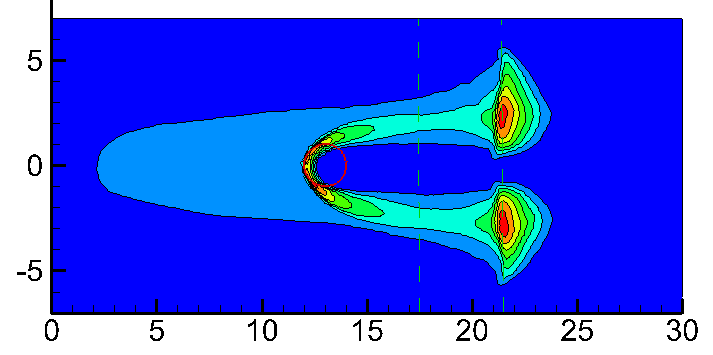} 
	\end{minipage}
	\begin{minipage}[t]{0.45\linewidth} 
		\centering 
		\includegraphics[width=1.2\textwidth]{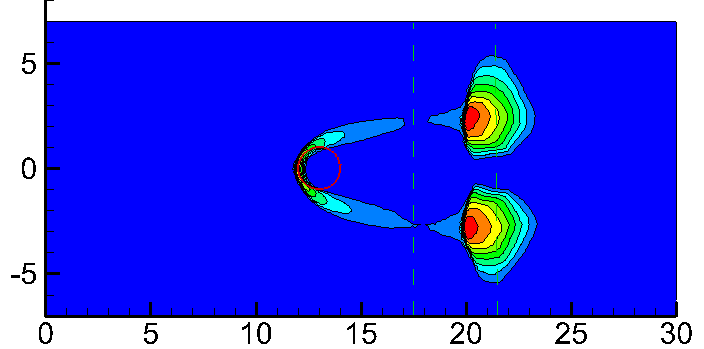} 
	\end{minipage}
	\hfill
	\begin{minipage}[t]{0.45\linewidth} 
		\centering 
		\includegraphics[width=1.2\textwidth]{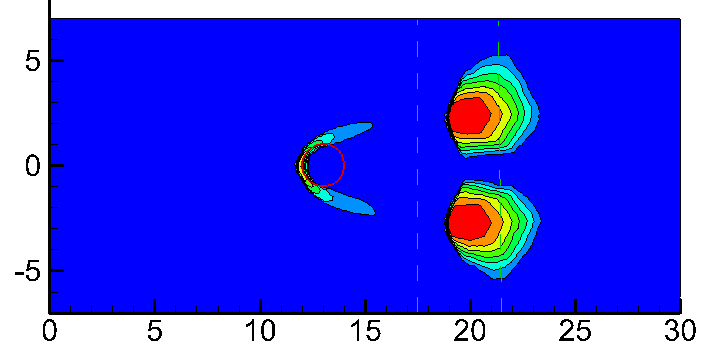} 
	\end{minipage}
	\begin{minipage}[t]{0.45\linewidth} 
		\centering 
		\includegraphics[width=1.2\textwidth]{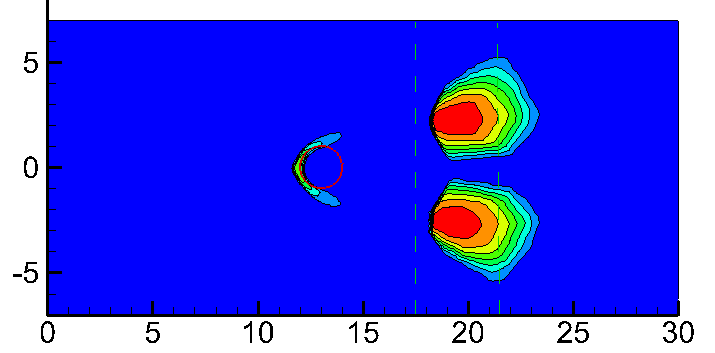} 
	\end{minipage}
	\hfill
	\begin{minipage}[t]{0.45\linewidth} 
		\centering 
		\includegraphics[width=1.2\textwidth]{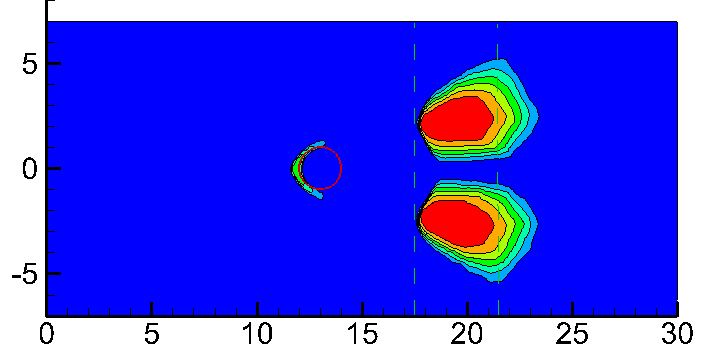} 
	\end{minipage}
	\centering
	\caption{Thickness contours of the avalanche at eight different 
		dimensionless times $t=3, 6, 9, 12, 15, 18, 21, 24$ for a 
		granular flow past a circular cone located on the inclined 
		plane and indicated by red solid lines with a maximum 
		dimensionless height of $H=1$. The transition zone from the 
		inclined plane to the horizontal plane lies between the two 
		green dashed lines.  }
	\label{fig::ex3}
\end{figure}

\begin{figure}
	\begin{minipage}[t]{0.45\linewidth}
		\centering 
		\includegraphics[width=1.2\textwidth,trim=80 50 80 430,clip]{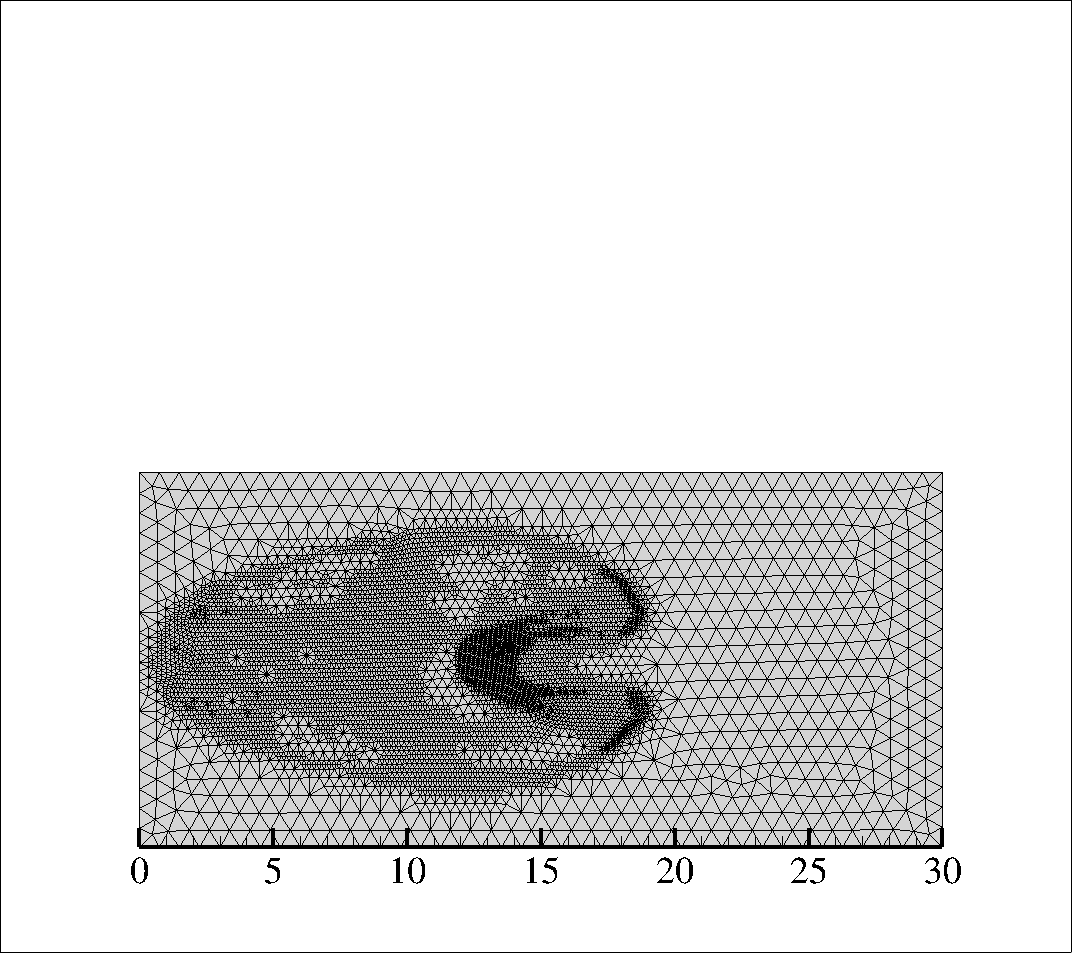}
	\end{minipage} 
	\hfill 
	\begin{minipage}[t]{0.45\linewidth} 
		\centering 
		\includegraphics[width=1.2\textwidth,trim=80 50 80 430,clip]{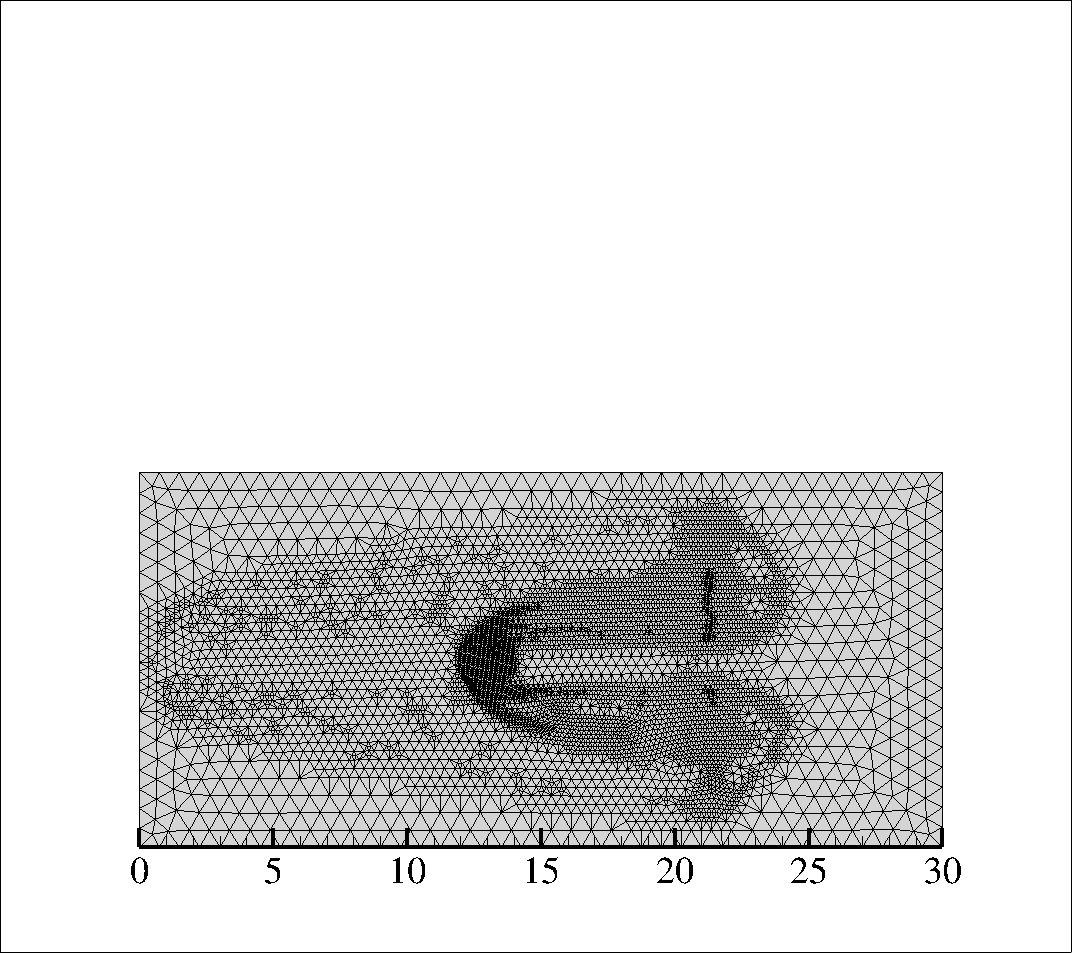}
	\end{minipage} 
	\begin{minipage}[t]{0.45\linewidth} 
		\centering 
		\includegraphics[width=1.2\textwidth,trim=80 50 80 430,clip]{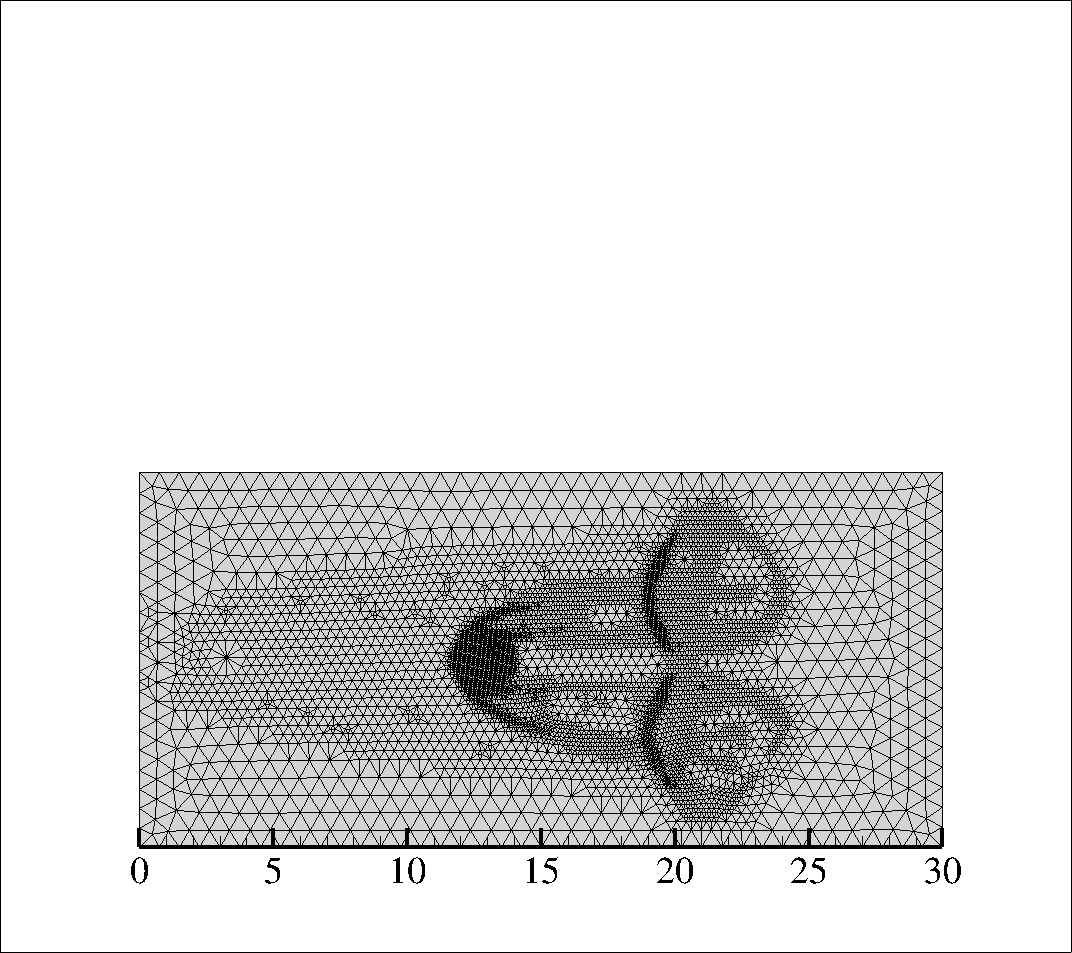} 
	\end{minipage}
	\hfill
	\begin{minipage}[t]{0.45\linewidth} 
		\centering 
		\includegraphics[width=1.2\textwidth,trim=80 50 80 430,clip]{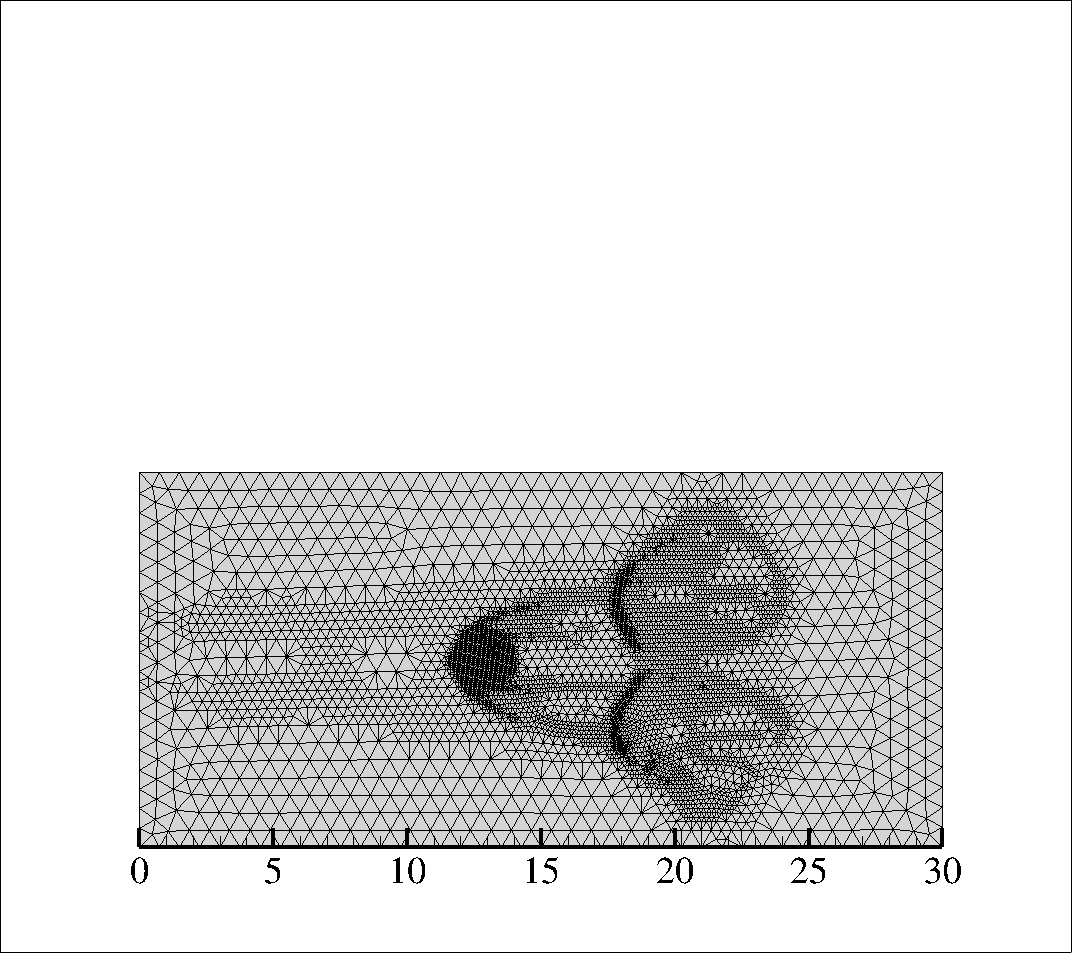}
	\end{minipage}
	\centering
	\caption{The triangular meshes at times $t=6, 12, 18, 24$ for the 
		flow slides down the inclined plane with a conical obstacle.}
	\label{fig::ex3mesh}	
\end{figure}


\section{Conclusions}
This paper is an attempt to solve the Savage-Hutter equations on
unstructred grids. Keeping in mind that the model is not rotational
invariant, it is more than encouraging for our numerical scheme to
attain numerical results with good agreement to the reference
solutions. It is indicated that the method is applicable to problems
as granular avalanche flows. The underlying reason why the scheme
works smoothly is still under investigation.


\section*{Acknowledgement}
This work is financially supported by the National Natural Science
Foundation of China(No. 11826207 and No. 11826208).

\bibliographystyle{plain}
\bibliography{mybibfile}

\end{document}